\documentclass{amsart}
\numberwithin{equation}{section}
\newtheorem{thm}{Theorem}[section]

\newtheorem{lem}[thm]{Lemma}

\usepackage{amssymb}
\usepackage{amsmath}
\usepackage{amsfonts}
\usepackage{youngtab}
\usepackage{graphicx}

\begin{document}
\title{An Elementary Proof of the Hook Content Formula}
\author{Graham Hawkes}

\begin{abstract}
In this paper we prove the Hook Content Formula (HCF) (and Hook Length Formula (HLF)) using induction. Instead of working with Young tableaux directly, we introduce a vector notation (sequences of these vectors represent (''single-vote") ballot sequences in the case of SYT, and ``multi-vote" ballot sequences--where the voter may choose any number of candidates--in the case of SSYT) \cite{mc} to aid in the inductive argument.
Next, we establish an identity (equation \eqref{alg}) which allows us to prove a formula that counts multi-vote ballot sequences. We demonstrate that, in the non-degenerate case (when this formula counts SSYTs) it coincides with the HCF. To do the latter, we borrow parts of a technique outlined by Wilson and Van Lint in their proof of the HLF \cite{van}. We then establish an identity, which is really a special case of equation \eqref{alg}, and show that the HLF follows from it. (Wilson and Van Lint prove this identity directly, and use it to prove the HLF in a similar manner.) We then note the appearance of an expression resembling the Weyl dimension formula \cite{weyl} and conclude with a combinatorial result.
\end{abstract}
\maketitle
\section{Introduction}
As far as proofs of the Hook Content Formula (HCF) it seems simplicity often takes a back seat to bijectivity. But a result so easily formulated should have a proof which is not overly complicated. More importantly, perhaps, a result so accessible should have a proof that is also accessible. The prerequisites for understanding the result of the HCF include such skills as counting squares and multiplying numbers. Of course, it is likely impossible to give a proof which uses only such tools, but our goal here is to come as close as possible. Namely, to present a proof that is simple, self-contained, and most importantly, accessible. 

The basic inductive approach is to first switch from thinking about semistandard Young tableaux to thinking about ``multi-vote" ballot sequences--on such ballots a voter may vote for as many candidates as he likes (including none). The problem is this: Suppose we have a group of candidates such that no two candidates have the same height. Further, suppose we know how many votes each candidate will receive as well as the total number of ballots that will be cast. Then, how many ways can said ballots be cast such that, at any point in the election, a taller candidate always has more (or the same) votes than a shorter one?

To answer this, we will look at the possible ways the last voter may have filled out his ballot, and investigate how each one of these hypothetical final ballots influences the possible ways the election could have proceeded up to that point. This will lead us to an obvious recursion that our answer must verify. We then present a formula which we claim \textit{will} satisfy this recursion. The rest, as the saying goes, is algebra.
\section{The Set-Up}

Suppose $d\geq 0$, and we are given $N\geq 0$, and a list of $(d+1)$ nonnegative integers, $n_0,\ldots, n_d$ (Indexing the $n_i$ starting on $0$ will make a number of computations neater). Let $v_1,\ldots,v_N$ be a list of $N$, $(d+1)$--dimensional vectors, each of which is composed of $1$'s and $0$'s and such that:
\begin{eqnarray*}
\sum_{k=1}^N v_k=\begin{bmatrix} n_0 \\ \vdots \\n_d \end{bmatrix}.\\
\end{eqnarray*}
Define $C(N,n_0,\ldots,n_d)$ to be the number of choices of such a list such that for each integer $L \in [1,N]$, the vector
$\sum\limits_{k=1}^L v_k$ is weakly decreasing from top to bottom. 

If there exists $i \in \{1,\ldots, d\}$ such that $n_{i-1} < n_i$, then $C(N,n_0,\ldots,n_d)=0$. If not, it follows that $n_0\geq \cdots\geq n_d$.
If, in addition, $n_d\neq 0$, then:
\begin{eqnarray*}
C(N,n_0,\ldots,n_d)= SSYT(N,(n_0,\ldots,n_d)'),
\end{eqnarray*}
where $SSYT(N,(n_0,\ldots,n_d)')$ refers to the number of semi-standard Young tableaux on a Young diagram of shape $(n_0,\ldots,n_d)'$ using integers from the set $\{1,\ldots,N\}$ (where the prime notation refers to partition conjugation).

In this case (i.e., $n_0\geq \cdots\geq n_d$ and $n_d\neq 0$) each $n_i \geq 1$, so, if $N\geq 1$ as well, we may write:
\begin{eqnarray*}
C(N,n_0,\ldots,n_d)=\sum_{(j_0,\ldots,j_d): j_i\in \{0,1\}} C(N-1,n_0-j_0,\ldots,n_d-j_d),
\end{eqnarray*}
where one should think of each possible vector of the form $\begin{bmatrix} j_0 \\ \vdots \\j_d \end{bmatrix}$ as a possible value for the final vector $v_N$, in the list $v_1,\ldots,v_N$.

\section{Counting $C(N,n_0,\ldots,n_d)$}
Let $d\geq 0$ and let $\vec{N}=(N,n_0,\ldots,n_d)$ be a $(d+2)$--tuple of nonnegative integers. If, for all $i \in \{1,\ldots, d\}$, we have $n_{i-1}+1\geq n_i$, define
\begin{eqnarray*}
\dagger(\vec{N})=1.
\end{eqnarray*}
Otherwise, define
\begin{eqnarray*}
\dagger(\vec{N})=0.
\end{eqnarray*}
Moreover, let:
\begin{eqnarray*}
F(N,n_0,\ldots,n_d)=\prod_{i=0}^d\frac{(N+i)!}{(N+i-n_i)!}\times \frac{V(m_0,\ldots,m_d)}{m_0!\cdots m_d!},
\end{eqnarray*}
where $V$ refers to the Vandermonde polynomial, and $m_i=n_i+d-i$ for each $i$. We use the conventions that $0!=1$, and for $n>0$, 
\begin{eqnarray*}
\frac{1}{(-n)!}= \lim_{k \rightarrow \infty} \prod_{m=n}^k \bigg(\frac{1}{-m}\bigg)=0.
\end{eqnarray*}
\begin{thm}\label{main}
Let $d\geq 0$ and let $\vec{N}=(N,n_0,\ldots,n_d)$ be a $(d+2)$--tuple of nonnegative integers. We have:
\begin{eqnarray*}
\dagger(\vec{N})=0 &\implies& C(\vec{N})=0\\
\dagger(\vec{N})=1 &\implies& C(\vec{N})=F(\vec{N}).
\end{eqnarray*}
\end{thm}
\begin{proof}
We prove the theorem by induction on $N$. Let $N=0$. First, suppose that $\dagger(\vec{N})=0$. It follows that some $n_i \neq 0$, whence, $C(\vec{N})=0$. On the other hand, suppose $\dagger(\vec{N})=1$. If each $n_i=0$, then $C(\vec{N})=1=F(\vec{N})$. Otherwise some $n_i>0$, and $C(\vec{N})=0$. Moreover, in this case, either $n_0>0$ so that the left hand factor of $F(\vec{N})$ vanishes, or, for some $i \in \{1,\ldots, d\}$, we have $n_{i-1}< n_i$ (i.e., $n_{i-1}+1=n_i$), so that the right hand factor of $F(\vec{N})$ vanishes. Regardless, $F(\vec{N})=0$. Thus, Theorem \ref{main} holds for $N=0$. 

Now suppose $N\geq 1$ and that \ref{main} holds for $N-1$. If $\dagger(\vec{N})=0$ it is clear that $C(\vec{N})=0$ as claimed, as $\begin{bmatrix} n_0 \\ \vdots \\n_d \end{bmatrix}$ itself is not weakly decreasing from top to bottom.

Therefore, assume $\dagger(\vec{N})=1$. First, suppose that for some $i \in \{1,\ldots, d\}$, we have $n_{i-1}< n_i$ (i.e., $n_{i-1}+1=n_i$). Clearly $C(\vec{N})=0$, and $F(\vec{N})=0$ as well, because the right hand factor of $F(\vec{N})$ vanishes. 
Hence, we may suppose that for each $i \in \{1,\ldots, d\}$, we have $n_{i-1}\geq n_i$. Since both $C(\vec{N})$ and $F(\vec{N})$ are invariant under the addition or removal of terminal 0's from $(n_0,\ldots,n_d)$, we may assume WLOG that $n_d\neq 0$. This implies that each $n_i \geq 1$, which, in conjunction with the fact that $N\geq 1$, allows us to write:

\begin{eqnarray*}
C(N,n_0,\ldots,n_d)&=&\sum_{(j_0,\ldots,j_d): j_i \in \{0,1\}} C(N-1,n_0-j_0,\ldots,n_d-j_d)\\
&=&\sum_{(j_0,\ldots,j_d): j_i \in \{0,1\}} F(N-1,n_0-j_0,\ldots,n_d-j_d),
\end{eqnarray*}
where the last equality follows by the inductive hypothesis and the fact that, for each $(j_0,\ldots,j_d)$, we have $\dagger(N-1,n_0-j_0,\ldots,n_d-j_d)=1$. Hence if we can establish the identity:
\begin{eqnarray}
F(N,n_0,\ldots,n_d)=\sum_{(j_0,\ldots,j_d): j_i \in \{0,1\}} F(N-1,n_0-j_0,\ldots,n_d-j_d). \label{Feq}
\end{eqnarray}
we are done.
\end{proof}

\section{A Useful Algebraic Result}
In order to establish \eqref{Feq} we will need the following lemma: 
\begin{lem}\label{prelem}
Write $\vec{X}=(X,x_0,\ldots,x_n)$ and define:
\begin{eqnarray*}
G(\vec{X},t)&=&\sum_{(j_0,\ldots,j_n): j_i \in \{0,1\}}\Bigg(\bigg[ \prod_{i=0}^n (X-x_i)^{1-j_i}x_i^{j_i}\bigg] V(x_0-j_0t,\ldots,x_n-j_nt)\Bigg).
\end{eqnarray*}
We have:
\begin{equation} 
G(\vec{X},t)=\Bigg[\prod_{r=0}^n (X-rt)\Bigg]V(x_0,\ldots,x_n)\label{alg},
\end{equation}
where $V$ refers to the Vandermonde polynomial.
\end{lem}

\begin{proof}
Before we start, we let $\vec{J}=(j_0,\ldots,j_n)$ and write: 
\begin{eqnarray*}
\psi(\vec{X},t,\vec{J})=\bigg[ \prod_{i=0}^n (X-x_i)^{1-j_i}x_i^{j_i}\bigg] V(x_0-j_0t,\ldots,x_n-j_nt),
\end{eqnarray*}
so that we have:
\begin{eqnarray*}
G(\vec{X},t)=\sum_{\vec{J}\in\{0,1\}^{n+1}}\psi(\vec{X},t,\vec{J}).
\end{eqnarray*}
\subsection{}

First, we show that $G(\vec{X},t)$ is antisymmetric with respect to transposition of any two of the variables $(x_0,\ldots,x_n)$. Indeed, fix $k$ and $l$ such that $0\leq k < l \leq n$ and let $\vec{J_{kl}}=(j_0,\ldots,j_{k-1},j_{k+1},\ldots,j_{l-1},j_{l+1},\ldots, j_n) \in \{0,1\}^{n-1}$ denote the values of $j_i$ in $\vec{J}$ for $i \neq k,l$. Then,

\begin{eqnarray*}
G(\vec{X},t)=\sum_{\vec{J}_{kl} \in \{0,1\}^{n-1}} \Bigg(\sum_{(j_k, j_l) \in \{0,1\}\times\{0,1\}} \psi(\vec{X},t,\vec{J}_{kl},j_k,j_l)\Bigg).
\end{eqnarray*}

$G(\vec{X},t)$ is antisymmetric because the expression inside the large parenthesis above is always antisymmetric. To see the latter, let $\vec{X}'$ be the vector obtained from $\vec{X}$ by switching $x_k$ and $x_l$. Then, for any fixed value of $\vec{J}_{kl}\in \{0,1\}^{n-1}$,
\begin{eqnarray*}
\psi(\vec{X},t,\vec{J}_{kl},0,0)&=&-\psi(\vec{X}',t,\vec{J}_{kl},0,0)\\
\psi(\vec{X},t,\vec{J}_{kl},0,1)&=&-\psi(\vec{X}',t,\vec{J}_{kl},1,0)\\
\psi(\vec{X},t,\vec{J}_{kl},1,0)&=&-\psi(\vec{X}',t,\vec{J}_{kl},0,1)\\
\psi(\vec{X},t,\vec{J}_{kl},1,1)&=&-\psi(\vec{X}',t,\vec{J}_{kl},1,1).
\end{eqnarray*}

\subsection{}

Now we show that $G(\vec{X},t)$ is homogenous of degree $\frac{n(n+1)}{2}$ with respect to the variables $(x_0,\ldots, x_n)$. Clearly any monomial in its monomial expansion must have degree at least $\frac{n(n+1)}{2}$ by antisymmetry. Suppose, therefore, that some monomial, $m$, in this expansion has degree larger than $\frac{n(n+1)}{2}$ in $(x_0,\ldots, x_n)$. It follows that $m$ has degree greater than $n$ (but no greater than $n+1$) in some variable $x_i$. We may assume, WLOG, this variable is $x_0$, that is, that $m$ has degree $n+1$ in $x_0$.

Let $\vec{J}_0=(j_1,\ldots,j_n) \in \{0,1\}^n$ denote the values of $j_i$ in $\vec{J}$ for $i\neq 0$, so that:
\begin{eqnarray}
G(\vec{X},t)=\sum_{\vec{J}_0 \in \{0,1\}^n} \Bigg(\sum_{j_0 \in \{0,1\}} \psi(\vec{X},t,\vec{J}_0,j_0)\Bigg). \label{hom}
\end{eqnarray}
$G(\vec{X},t)$ has no monomials of degree $n+1$ in $x_0$ because the expression inside the parenthesis above never has any. Indeed, fix $\vec{J}_0=(j_1,\ldots,j_n)$.
Then the entire degree $n+1$ (with respect to $x_0$) part of $\psi(\vec{X},t,\vec{J}_0,0)$ is given by:
\begin{eqnarray*}
-x_0\Bigg[\prod_{i=1}^n (X-x_i)^{1-j_i}x_i^{j_i}\Bigg]x_0^n \Bigg[\sum_{\sigma \in S_n}(-1)^{sgn(\sigma)}\Bigg(\prod_{i=1}^n (x_i-j_it)^{n-\sigma(i)}\Bigg)\Bigg],
\end{eqnarray*}
whereas the entire degree $n+1$ (with respect to $x_0$) part of $\psi(\vec{X},t,\vec{J}_0,1)$ is given by:
\begin{eqnarray*}
x_0\Bigg[\prod_{i=1}^n (X-x_i)^{1-j_i}x_i^{j_i}\Bigg]x_0^n \Bigg[\sum_{\sigma \in S_n}(-1)^{sgn(\sigma)}\Bigg(\prod_{i=1}^n (x_i-j_it)^{n-\sigma(i)}\Bigg)\Bigg].
\end{eqnarray*}
Hence the degree $n+1$ (with respect to $x_0$) of the expression inside the parenthesis in \eqref{hom} is 0, so it has no monomials of degree $n+1$ in $x_0$. We have established that each monomial in the monomial expansion of $G(\vec{X},t)$ has total degree $\frac{n(n+1)}{2}$ in the variables $(x_0,\ldots, x_n)$.

\subsection{}

Since $G(\vec{X},t)$ is antisymmetric, homogenous of degree $\frac{n(n+1)}{2}$, with respect to $(x_0,\ldots,x_n)$, it is divisible by $V(x_0,\ldots,x_n)$, and the quotient has degree $0$ in $(x_0,\ldots,x_n)$. That is, we may write:
\begin{eqnarray*}
G(\vec{X},t)=H(X,t)\times V(x_0,\ldots,x_n),
\end{eqnarray*}
for a function $H$, that only depends on the two variables, $X$ and $t$. In this section, we compute $H(X,t)$.

For $\vec{J} \in \{0,1\}^{n+1}$, $\sigma \in S_{n+1}$, write:
\begin{eqnarray*}
\Pi (\vec{J}) &=& \prod_{i=0}^n (X-x_i)^{1-j_i}x_i^{j_i}\\
V_{\sigma}(\vec{J})&=&(-1)^{sgn(\sigma)}(x_0-j_0t_0)^{n-\sigma(0)}\cdots(x_n-j_nt_n)^{n-\sigma(n)},
\end{eqnarray*}
so that we have:
\begin{eqnarray*}
G(\vec{X},t)= \sum_{\vec{J} \in \{0,1\}^{n+1}} \Bigg[\sum_{\sigma \in S_{n+1}} \Pi (\vec{J}) V_{\sigma}(\vec{J})\Bigg].
\end{eqnarray*}
We make the following definitions:
\begin{enumerate}
\item{Let $p \in \mathbb{Z}[X,t,x_0,\ldots,x_n]$.  Consider $p$ as a polynomial in $x_0,\ldots,x_n$ with coefficients in $\mathbb{Z}[X,t]$. Define $\delta(p) \in \mathbb{Z}[X,t]$ to be the coefficient  of $x_0^n\cdots x_n^0$ in $p$.}
\item{Let $p_i \in \mathbb{Z}[X,t,x_i]$.  Consider $p_i$ as a polynomial in $x_i$ with coefficients in $\mathbb{Z}[X,t]$. Define $\delta_i(p_i) \in \mathbb{Z}[X,t]$ to be the coefficient  of $x_i^{n-i}$ in $p_i$.}
\item{Define $\delta_{ij}=1$ if $i=j$, and $\delta_{ij}=0$ if $i\neq j$.  (Kronecker delta function.)}
\end{enumerate}  
Using these definitions, we may write:
\begin{eqnarray*}
H(X,t)=\delta(G(\vec{X},t))= \sum_{\vec{J} \in \{0,1\}^{n+1}} \Bigg[\sum_{\sigma \in S_{n+1}} \delta \bigg(\Pi(\vec{J})V_{\sigma}(\vec{J})\bigg)\Bigg].
\end{eqnarray*}

Since $\Pi(\vec{J})$ is linear in each $x_i$, it follows that $\delta \big(\Pi(\vec{J})V_{\sigma}(\vec{J})\big)=0$ unless $\sigma(i) \in \{i,i+1\}$ for each $i$. The only such permutation is the identity, so,
\begin{eqnarray*}
H(X,t)&=&\sum_{\vec{J} \in \{0,1\}^{n+1}} \delta \bigg(\Pi(\vec{J})(x_0-j_0t_0)^n\cdots(x_n-j_nt_n)^0\bigg)\\
&=&\sum_{\vec{J} \in \{0,1\}^{n+1}} \delta \bigg(\prod_{i=0}^n \Big[(X-x_i)^{1-j_i}x_i^{j_i}(x_i-j_i t)^{n-i}\Big]\bigg)\\
&=&\sum_{\vec{J} \in \{0,1\}^{n+1}} \Bigg[\prod_{i=0}^n \delta_i \Big((X-x_i)^{1-j_i}x_i^{j_i}(x_i-j_i t)^{n-i}\Big)\Bigg]\\
&=&\sum_{\vec{J} \in \{0,1\}^{n+1}}\Bigg[ \prod_{i=0}^n X^{(\delta_{0j_i})}\big((i-n)t\big)^{(\delta_{1j_i})}\Bigg].
\end{eqnarray*}
From this expression, we see that each monomial in the monomial expansion of $H(X,t)$ has total degree $n+1$ in the variables $X,t$. Thus we may write:
\begin{align*}
H(X,t)=\sum_{m=0}^{n+1} a_m X^{(n+1-m)}t^m, \text{\,\,\,\,\,\,\,\,where,\,\,\,\,\,\,\,\,} a_m&=&\sum_{\vec{J}:\sum (j_i)=m} \Bigg[ \prod_{j_i=1} (i-n)\Bigg]\\
&=&\sum_{I \in {\{0,\ldots,n\} \choose m}} \Bigg[ \prod_{i\in I} (-i)\Bigg],
\end{align*}
and where $ {\{0,\ldots,n\} \choose m}$ is the set of all $m$ element subsets of $\{0,\ldots,n\}$. 
Moreover, if we write out the following expansion:
\begin{align*}
\prod_{r=0}^n (X-rt)=\sum_{m=0}^{n+1} b_m X^{(n+1-m)}t^m,  \text{\,\,\,\,\,\,\,we find:\,\,\,\,\,\,\,} b_m&=&\sum_{I \in {\{0,\ldots,n\} \choose m}} \Bigg[ \prod_{i\in I} (-i)\Bigg],\\
&\implies& H(X,t)=\prod_{r=0}^n (X-rt),
\end{align*}
and the lemma is complete.
\end{proof}

\section{Conclusion of Theorem \ref{main}}
We apply \eqref{alg} with $n=d$, $X=N+d$, $x_i=m_i$, and $t=1$, noting that $([N+d]-m_i)$ may be replaced by $(N+i-n_i)$ by the definition of $m_i$.  This gives:
\begin{eqnarray*}
&&\frac{(N+d)!}{(N-1)!} V(m_0,\ldots,m_d)\\
&=&\sum_{(j_0,\ldots,j_d)\in \{0,1\}^{d+1}} \Bigg(\Bigg[\prod_{i=0}^d(N+i-n_i)^{1-j_i}m_i^{j_i} \Bigg]V(m_0-j_0,\ldots, m_d-j_d)\Bigg).
\end{eqnarray*}
Multiplying both sides by:
\begin{align*}
\Bigg[\prod_{i=0}^d\frac{(N+i-1)!}{(N+i-n_i)!}\Bigg]\frac{1}{m_0!\cdots m_d!},
\end{align*}
we have:
\begin{eqnarray*}
&&\Bigg[\prod_{i=0}^d\frac{(N+i)!}{(N+i-n_i)!}\Bigg]\frac{V(m_0,\ldots,m_d)}{m_0!\cdots m_d!}\\
&=&\sum_{(j_0,\ldots,j_d)\in \{0,1\}^{d+1}} \Bigg(\Bigg[\prod_{i=0}^d\frac{(N+i-1)!}{(N+i-n_i)!}(N+i-n_i)^{1-j_i}m_i^{j_i} \Bigg]\frac{V(m_0-j_0,\ldots,m_d-j_d)}{m_0!\cdots m_d!}\Bigg)\\
&=&\sum_{(j_0,\ldots,j_d)\in \{0,1\}^{d+1}} \Bigg(\Bigg[\prod_{i=0}^d\frac{(N+i-1)!}{(N+i-n_i-1+j_i)!}\Bigg]\frac{V(m_0-j_0,\ldots,m_d-j_d)}{(m_0-j_0)!\cdots (m_d-j_d)!}\Bigg).\\
\end{eqnarray*}
Equating the first and third lines above gives:
\begin{eqnarray*}
F(N,n_0,\ldots,n_d)=\sum_{(j_0,\ldots,j_d)\in \{0,1\}^{d+1}} F(N-1,n_0-j_1,\ldots,n_d-j_d).
\end{eqnarray*}
This establishes Theorem \ref{main}.

\section{The Hook Content Formula}
\textbf{In accordance with our indexing for the $n_i$, we will also index the rows and columns of a Young diagram beginning on $0$. We will refer to row $i$, column $i$, candidate $i$, etc., with respect to this indexing (i.e., to refer to the $(i+1)^{st}$ row, column, etc.) throughout the rest of the paper.}

Let $\mu$ be a Young diagram, such that $\mu'=(n_0,\ldots,n_d)$. Then, by definition, we must have $n_0\geq \cdots \geq n_d \geq 1$. As noted earlier, this implies that if $N$ is any nonnegative integer and we let $\vec{N}=(N,n_0,\ldots,n_d)$, then $SSYT(N,\mu)=C(\vec{N})$. Moreover, since $\dagger(\vec{N})=1$, it follows by Theorem \ref{main} that $C(\vec{N})=F(\vec{N})$, whence:
\begin{eqnarray}
SSYT(N,\mu)=\prod_{i=0}^d\frac{(N+i)!}{(N+i-n_i)!}\times \frac{V(m_0,\ldots,m_d)}{m_0!\cdots m_d!}. \label {form}
\end{eqnarray}
(Again, we use the notation $m_i=n_i+d-i$).

On the other hand, the Hook Content Formula states that:
\begin{eqnarray*}
SSYT(N,\mu)=\prod_{x_{ij} \in \mu} \frac {|x_{ij}|}{|h_{ij}|}=\Bigg[\prod_{{x_{ij}}\in \mu}(N+i-j)\Bigg] \Bigg[\prod_{{x_{ij}}\in \mu}\frac{1}{|h_{ij}|}\Bigg],
\end{eqnarray*}
\cite{stan} where $h_{ij}$ refers to the hook through $x_{ij}$, $|h_{ij}|$ the length of this hook, and $|x_{ij}|$ the hook content of $x_{ij}$. It follows from the fact that $\mu'=(n_0,\ldots,n_d)$ that:
\begin{eqnarray*}
\prod_{{x_{ij}}\in \mu}(N+i-j)=\prod_{i=0}^d\frac{(N+i)!}{(N+i-n_i)!},
\end{eqnarray*}
so, the hook content formula will follow from $\eqref{form}$ if we show that:

\begin{eqnarray*}
(*)\prod_{{x_{ij}}\in \mu}|h_{ij}|=\frac{m_0!\cdots m_d!}{V(m_0,\ldots,m_d)}, \text{ or equivalently, } \prod_{{x_{ij}}\in \lambda}|h_{ij}|=\frac{m_0!\cdots m_d!}{V(m_0,\ldots,m_d)},
\end{eqnarray*}
\cite{van} for $\lambda=(n_0,\ldots,n_d)=\mu'$, since the product of hook lengths is invariant under conjugation.

To demonstrate the latter equality, first note that it may be rewritten as:

\begin{equation*}
\prod_{{x_{ij}}\in \lambda}|h_{ij}|=\prod_{i=0}^d\Bigg[ \frac{m_i!} {\prod\limits^{d}_{j=i+1}(m_i-m_j)}\Bigg]. \label{p}
\end{equation*}

To establish the equation above, thereby proving $(*)$, we show that, for each $i$, the product of the hook lengths of the squares in row $i$ of $\lambda$ (denote $\lambda_i$) is given by:
\begin{equation}
\prod_{{x_{ij}}\in \lambda_i}|h_{ij}|=\frac{m_i!}{\prod\limits^{d}_{j=i+1}(m_i-m_j)}. \label{i}
\end{equation}

\large
First, note that, for any $j$ such that $i < j \leq d$, the value of $m_i-m_j$ does not coincide with the hook length of any of the squares in $\lambda_i$. To see this, let $\overline{h_{in_j}}$ be the hook obtained by removing from $h_{in_j}$ all the squares below $x_{j-1n_j}$. The path from $x_{d0}$ to $x_{in_i}$ along $h_{i0}$ includes $m_i$ squares. It follows that the path from $x_{d0}$ to $x_{in_i}$ that begins along $h_{j0}$ and concludes along $\overline{h_{in_j}}$ must also include $m_i$ squares. Moreover, the path from $x_{d0}$ to $x_{jn_j}$ along $h_{j0}$ includes $m_j$ squares. From this it follows that the length of $\overline{h_{in_j}}$ is given by $m_i-m_j$. One easily observes that for $k \leq n_j$, $|h_{ik}|>|\overline{h_{in_j}}|$, and for $k > n_j$, $|h_{ik}|<|\overline{h_{in_j}}|$. Hence, no square in row $i$ of $\lambda$ has hook length equal to $|\overline{h_{in_j}}|=m_i-m_j$.
\normalsize
\begin{center}
\includegraphics[scale=0.45]{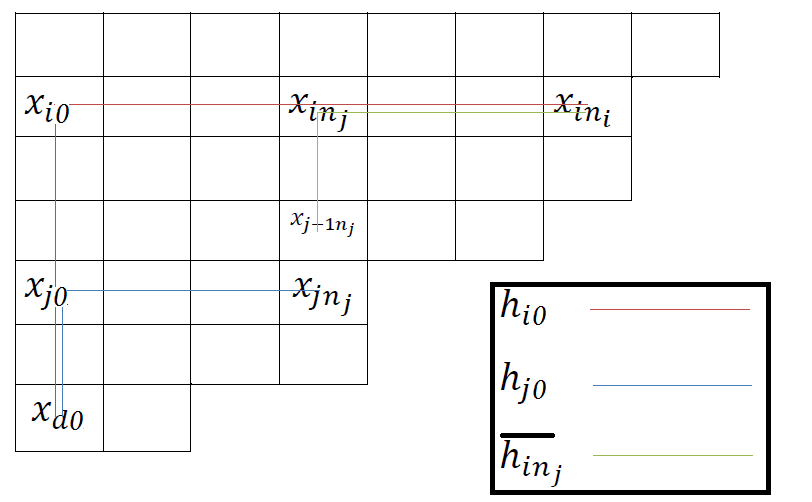}
\end{center}

Let $H_i=\{|h_{ij}|:x_{ij}\in \lambda_i\}$, be the set of hook lengths in row $i$ of $\lambda$ (each hook in a row has a distinct length), let $K_i=\{(m_i-m_j): i<j\leq d\}$, and let $M_i=\{1,\ldots, m_i\}$. Now $H_i\subseteq M_i$, $K_i\subseteq M_i$, and:
\begin{eqnarray*}
\#(H_i)+\#(K_i)=(n_i)+(d-i)=m_i=\#(M_i).
\end{eqnarray*}
Further, by the argument above $H_i$ and $K_i$ are disjoint, so $H_i\dot\cup K_i=M_i$, whence:
\begin{eqnarray*}
\prod_{{x_{ij}}\in \lambda_i}|h_{ij}|=\prod_{h\in H_i}(h)=\frac{\prod_{m\in M_i}(m)}{\prod_{k\in K_i}(k)}=\frac{m_i!}{\prod\limits^{d}_{j=i+1}(m_i-m_j)},
\end{eqnarray*}
and \eqref{i} has been proven. This establishes $(*)$, and The Hook Content Formula now follows.
\section{Relation to The Hook Length Formula}

Suppose in our definition of $C(\vec{N})$ we impose the additional requirement that each vector $v_k$ must contain exactly one entry equal to $1$, and refer to this new count as $C^{\ast}(\vec{N})$. In other words, $C^{\ast}(\vec{N})$ gives the number of ways to add $N$ $d$--dimensional vectors each composed of $(d-1)$ $0$'s and one 1, so that each partial sum is weakly decreasing by coordinate, and so that the final sum is equal to $n_i$ in each coordinate. Of course, unless $N=\sum(n_i)$, and the $n_i$ are weakly decreasing by coordinate, $C^*(\vec{N})=0$. On the other hand, if both these conditions hold, and $n_d \neq 0$, then $C(\vec{N})$ gives the number of standard Young tableaux on a Young diagram of shape $(n_0,\ldots, n_d)'$. In fact, the proof given in this paper can easily be adapted to a proof of the Hook Length Formula.

First, note that $C^*(\vec{N})$ satisfies the exact same recursion as $C(N)$. That is, if $n_0\geq \cdots\geq n_d$, $n_d\neq 0$, and $N\geq 1$, we may write:
\begin{eqnarray*}
C^*(N,n_0,\ldots,n_d)=\sum_{i=0}^n C^*(N-1,n_0,\ldots,n_i-1,\ldots.n_d), 
\end{eqnarray*}
where we have taken advantage of the fact that the only summands on the right hand side which do not vanish are those for which exactly one $j_i=1$.

If we set the degree $n$ (with respect to the variable $X$) part of the expression defining $G(\vec{X},t)$ in Lemma \ref{prelem} equal to the degree $n$ (with respect to $X$) part of the expression found in equation \eqref{alg} we have the following identity:
\begin{align*}
\sum_{i=0}^n \bigg[(X^n)(x_i) V(x_0,\ldots,x_i-t,\ldots, x_n)\bigg]&+&\sum_{i=0}^n \bigg[(X^n)(-x_i) V(x_0,\ldots,x_i,\ldots, x_n)\bigg] \\
&=&\sum_{r=0}^n \bigg[(X^n)(-rt)V(x_0,\ldots,x_n)\bigg]\\
\implies\sum_{i=0}^n \bigg((x_i)V(x_0,\ldots,x_i-t,\ldots, x_n)\bigg)&=&\Bigg[\sum_{i=0}^n (x_i -it) \Bigg]V(x_0,\ldots,x_n)
\end{align*}
\cite[p. 163]{van}.
If we make the substitutions, $n=d$, $x_i=m_i$, and $t=1$, and let $N=\sum(n_i)$, then after multiplying by $(N-1)!/(m_0!\cdots m_d!)$ this becomes:
\begin{eqnarray*}
\sum_{i=0}^d \bigg( (N-1)!\times \frac{ V(x_0,\ldots,m_i-1,\ldots, m_d)}{(m_0!\cdots (m_i-1)! \cdots m_d!)}\bigg)=N! \times \frac{V(m_1,\ldots,m_d)}{m_1!\cdots m_d!}.
\end{eqnarray*}
Define $F^*(N)$ to be the expression on the right hand side of the equation above. We see immediately that $F^*(N)$ satisfies the same recursion as $C^*(\vec{N})$ whenever $n_0\geq \cdots\geq n_d\geq 1$, and $N=\sum(n_i)$.
By the same inductive arguments as in the original proof, this implies that when $(n_0,\ldots,n_d)$ is a partition of $N$, $F^*(\vec{N})$ gives the number of SYT of shape $(n_0,\ldots,n_d)'$, i.e.:
\begin{eqnarray}
C^*(\vec{N})=N!\times \frac{V(m_1,\ldots,m_d)}{m_1!\cdots m_d!} \label{hookl}
\end{eqnarray} 
\cite[p. 164]{van}. Since the right hand factor is the reciprocal of the product of the hook lengths of $\mu=(n_0,\ldots,n_d)'$ by the arguments in the last section, we have proven the Hook Length Formula.

\section{A Combinatorial Result}
For any partition $\mu$ with $\mu'=(n_0,\ldots,n_d)$:

\begin{enumerate}
\item{ Let $T(N,\mu)=SSYT(N,\mu)$, and $T^*(\mu)=SYT(\mu)$.}
\item{Let $\textbf{T}^C(N,\mu)$ be the number of ways to fill a Young diagram of shape $\mu$ with strictly increasing columns of positive integers such that the entries in column $i$ are not greater than $N+i$.}
\item{ Let $\textbf{T}^R(N,\mu)$ be the number of ways to fill a Young diagram of shape $\mu$ with weakly increasing rows of positive integers such that the entries in row $j$ are greater than $j$, but not greater than $N$.}
\item{Let $\textbf{T}^*(\mu)$ be the number of ways to fill a Young diagram of shape $\mu$ with one of each of the numbers $\{1,\ldots,|\mu|\}$, such that the numbers are increasing down each column.}
\item{Define $T_0(\mu)=T(n_0,\mu)$, $\textbf{T}^C_0(\mu)=\textbf{T}^C(n_0,\mu)$, and $\textbf{T}^R_0(\mu)=\textbf{T}^R(n_0,\mu)$. (This is the least nonzero value of $T(N,\mu)$, etc., over all $N\geq 0$.)}
\item{Denote the set that each of these functions is counting with the prefix $S$. E.g., $ST(N,\mu)$ is the set of elements that are counted by $T(N,\mu)$.}
\end{enumerate}
\textbf{Remark.} Each element of $ST(N,\mu)$ appears as an element of both $S\textbf{T}^C(N,\mu)$ and $S\textbf{T}^R(N,\mu)$, in fact: $ST(N,\mu)=S\textbf{T}^C(N,\mu)\cap S\textbf{T}^R(N,\mu)$.
Each element of $ST^*(\mu)$ appears as an element of $S\textbf{T}^*(\mu)$. In fact, the former is equivalent to the set of (single-vote) ballot sequences with result $\mu'=(n_0,\ldots,n_d)$, and the latter to the set of \textit{all} possible sequences of votes by which that result could occur in a single-vote election. \\

Henceforth, assume $N \geq n_0$. We make the following computations:
\begin{align*} 
&\textbf{T}^C(N,\mu)=\prod_{i=0}^d {N+i \choose n_i}& &\textbf{T}_0^R(\mu')=\prod_{i=0}^d {n_i+d-i \choose n_i}& & \textbf{T}^*(\mu)={ |\mu| \choose n_0,\ldots, n_d}.&
\end{align*}
For $\mu=(n_0,\ldots,n_d)'$, define:
\begin{eqnarray*}
P(\mu)=\frac{n_0!\cdots n_d!}{\prod_{x_{ij}\in \mu} |h_{ij}|}=\frac{V(m_0,\ldots,m_d)}{m_0! \cdots m_d!} \times (n_0!\cdots n_d!) =\left[ \prod_{0\leq i<j \leq d} \frac{n_i-n_j+j-i}{n_i+j-i}\right].
\end{eqnarray*}
Note that \eqref{form} may now be rewritten as:
\begin{eqnarray*}
T(N,\mu)=\prod_{i=0}^d {N+i \choose n_i}P(\mu)=\textbf{T}^C(N,\mu)P(\mu),
\end{eqnarray*}
and further, that it follows from \eqref{hookl} that:
\begin{eqnarray*}
T^*(\mu)={|\mu| \choose n_0, \ldots, n_d}P(\mu)=\textbf{T}^*(\mu)P(\mu).
\end{eqnarray*}

We now rewrite $P(\mu)$ as the product of two expressions as follows: 
\begin{eqnarray*}
P(\mu)=\left[ \prod_{0\leq i<j \leq d} \left(\frac{n_i-n_j+j-i}{j-i}\right)\right] \left[\prod_{0\leq i<j \leq d} \left( \frac{j-i}{n_i+j-i}\right) \right].
\end{eqnarray*}
We let $\mathcal{W}(\mu)$ refer to the left hand factor,\footnote{In the next few lines we prove that $\mathcal{W}(\mu)=T(d+1,\mu')=T_0(\mu')$. Alternatively, one may recall that the Schur polynomial $s_{\mu'}(x_0,\ldots,x_d)$ evaluated at $(1,\ldots,1)$ is exactly $T_0(\mu')$. The result then follows after calculating that $s_{\mu'}(1,\ldots,1)=\mathcal{W}(\mu)$ \cite[p. 66]{rep}. The expression for $\mathcal{W}(\mu)$ often appears in representation theory as a form of Weyl's dimension formula \cite{weyl}.} and $\mathcal{R}(\mu)$ to the reciprocal of the right hand factor, so that we may write $P(\mu)=\frac{\mathcal{W}(\mu)}{ \mathcal{R}(\mu)}$. Moreover, we may rewrite:
\begin{eqnarray*}
\mathcal{W}(\mu)= \left(\frac{m_0! \cdots m_d!}{\prod \limits_{i=0}^d (d-i)!}\right)\Bigg[\frac{V(m_0,\ldots,m_d)}{m_0! \cdots m_d!}\Bigg]=\Bigg[\prod_{i=0}^d\frac{(n_i+d-i)!}{(d-i)!}\Bigg] \Bigg[\frac{V(m_0,\ldots,m_d)}{m_0!\cdots m_d!}\Bigg].
\end{eqnarray*}

Let's examine the expression on the right. Its right hand factor is the familiar reciprocal of the product of the hook lengths of $\mu$, which, by symmetry, is also the reciprocal of the product of the hook lengths of $\mu'$. Moreover, its left hand factor is seen to be the product of the hook contents of a Young diagram with parameters $(d+1,\mu')$. Thus $\mathcal{W}(\mu)=T(d+1,\mu')=T_0(\mu')$.
Furthermore, note that:
\begin{eqnarray*}
\mathcal{R}(\mu)=\left[\prod_{0\leq i<j \leq d} \left( \frac{n_i+j-i}{j-i}\right) \right]=\prod_{i=0}^d \Bigg[\frac{(n_i+d-i)!}{(d-i)!}\Bigg]={n_i+d-i \choose n_i}=\textbf{T}_0^R(\mu').
\end{eqnarray*}
Putting this all together, we have:
\begin{eqnarray*}
T_0(\mu')=\textbf{T}_0^R(\mu')P(\mu).
\end{eqnarray*}

Combining this with the earlier results, we have:
\begin{align*}
\frac{T(N,\mu)}{\textbf{T}^C(N,\mu)}=\frac{T_0(\mu')}{\textbf{T}_0^R(\mu')}=\frac{T^*(\mu)}{\textbf{T}^*(\mu)}=P(\mu).
\end{align*}
That is, for any partition $\mu=(n_0,\ldots, n_d)$ and any integer $N \geq n_0$ we have:
\begin{thm}
If $a \in S\textbf{T}^C(N,\mu)$, $b \in S\textbf{T}_0^R(\mu')$, and $c \in S\textbf{T}^*(\mu)$, the following three probabilities are the same:
\begin{align*}
&Pr(a \in ST(N,\mu))& &=& &Pr(b \in ST_0(\mu'))& &=& &Pr(c \in ST^*(\mu)).&
\end{align*}
And they each coincide with the expression:
\begin{eqnarray*}
P(\mu) =\frac{n_0!\cdots n_d!}{\prod_{x_{ij}\in \mu} |h_{ij}|}.
\end{eqnarray*}
\end{thm}


\begin{thebibliography}{5}

\bibitem{weyl} H. Weyl, \emph{The Classical Groups}, Princeton University Press, Princeton, 1939, Ch. 7.\\

\bibitem{symm} Laurent Manivel, \emph{Symmetric Functions, Schubert Polynomials, and Degeneracy Loci}, American Mathematical Soc., 2001.\\

\bibitem{seq}
N. J. A. Sloane, On-line Encyclopedia of Integer Sequences, http://oeis.org; Sequences A117506 and A210391.\\

\bibitem{mc}
P. A. MacMahon, \emph{Combinatory Analysis}, I., Cambridge Univ. Press, 1915.\\

\bibitem{rep} P. Etingof, O. Golberg, S. Hensel, T. Liu, A. Schwendner, D. Vaintrob, and E. Yudovina. \emph{Introduction to Representation Theory}, volume 59 of Student Mathematical Library. AMS, 2011, pp.63-67.\\

\bibitem{van} R.H. van Lint and R.M. Wilson, \emph{A Course in Combinatorics}, 2nd ed., Cambridge University Press, Cambridge, 2001, pp. 162-166.\\


\bibitem{stan} Richard P. Stanley, \emph{Enumerative combinatorics. Vol. 2}, Cambridge Studies in Advanced Mathematics, vol. 62, Cambridge University Press, Cambridge, 1999.\\









\end{thebibliography}
\end{document}